\documentclass[a4paper,10pt]{article}
\usepackage{multicol}
\usepackage{enumerate}
\usepackage{amssymb} \usepackage{amsfonts} \usepackage{amsmath}
\usepackage{amsthm} \usepackage{epsfig} 
\usepackage{amscd} \usepackage[all]{xy}
\usepackage{wrapfig}
\usepackage{amsbsy}
\usepackage{array}
\usepackage{mdwlist}
\usepackage{mathrsfs}
\usepackage{subfig}
\usepackage{sseq}
\usepackage{color}

\usepackage[firstinits=true,backend=bibtex,isbn=false,url=false,doi=false]{biblatex}
\bibliography{refs2}
\usepackage{caption}
\usepackage[utf8]{inputenc}
\usepackage[T1]{fontenc}
\usepackage{authblk}
\usepackage[margin=1.5in]{geometry}

\newtheorem{theorem}{Theorem}[section]
\newtheorem{lemma}[theorem]{Lemma}
\newtheorem{prop}[theorem]{Proposition}
\newtheorem{corollary}[theorem]{Corollary}
\newtheorem{property}[theorem]{Property}

\theoremstyle{definition}

\theoremstyle{remark}
\newtheorem{remark}[theorem]{Remark}

\newcommand{\p} {\ensuremath {\mathbb{P}}}
\newcommand{\E} {\ensuremath {\mathbb{E}}}

\newcommand{\N} {\ensuremath {\mathbb{N}}}
\newcommand{\R} {\ensuremath {\mathbb{R}}}

\newcommand{\I} {\ensuremath {\mathbb{I}}}
\newcommand{\F} {\ensuremath {\mathscr{F}}}

\newcommand{\Nn} {\ensuremath {\mathscr{Nn}}}

\newcommand{\A} {\ensuremath {\mathscr{A}}}
\newcommand{\X} {\ensuremath {\mathscr{X}}}

\newcommand{\B} {\ensuremath {\mathscr{B}}}
\newcommand{\M} {\ensuremath {\mathscr{M}}}

\newcommand{\Ci} {\ensuremath {\mathscr{C}}}
\newcommand{\mo} {\ensuremath {\mathscr{P}}}

\newcommand{\Qi} {\ensuremath {\mathscr{Q}}}

\title 
{Asymptotic equivalence of discretely observed diffusion processes and their Euler scheme: small variance case}

\author{Ester Mariucci\thanks{Ester.Mariucci@imag.fr}}
\affil{\it Laboratoire LJK, Universit\'e Joseph Fourier UMR 5224\\
      \it 51, Rue des Math\'ematiques, Campus de Saint Martin d'H\`eres\\
       \it BP 53 38041 Grenoble Cedex 09}

\date{}

\begin{document}
\maketitle
\begin{abstract}
This paper establishes the global asymptotic equivalence, in the sense of the Le Cam $\Delta$-distance, between scalar diffusion models with unknown drift function and small variance on the one side, and nonparametric autoregressive models on the other side. The time horizon $T$ is kept fixed and both the cases of discrete and continuous observation of the path are treated. We allow non constant diffusion coefficient, bounded but possibly tending to zero. The asymptotic equivalences are established by constructing explicit equivalence mappings.  
\end{abstract}

\vspace{0.5cm}

\vspace{0.4cm}

\section{Introduction}
Diffusion processes obtained as small random perturbations of deterministic dynamical systems have been widely studied and have proved fruitful in applied problems (see e.g. \cite{FR}). Among other subjects, they have been applied to contingent claim pricing, see \cite{UM} and the references therein, to filtering problems, see e.g. \cite{pi1,pi2} \and more recently to epidemic data \cite{guy}. From a statistical point of view, these models have first been considered by Kutoyants \cite{k84} in the framework of continuous observation on a fixed time interval $[0,T]$. However, statistical inference for discretely observed diffusion processes has first been treated several years after, see \cite{g90}. In a nonparametric framework we may quote \cite{k284}, among many others.

In this paper we consider the problem of estimating the drift function $f$ associated with a scalar diffusion process $(y_t)$ continuously or discretely observed on a time interval $[0,T]$, with $T<\infty$ kept fixed. More precisely, we consider the one-dimensional diffusion process $(y_t)$ given by
\begin{equation}\label{eq:y}
 dy_t=f(y_t)dt+\varepsilon\sigma(y_t)dW_t, \quad t\in [0,T],\quad y_0=w\in \R,
\end{equation}
where $(W_t)_{t\geq 0}$ is a standard $(\A_t)_{t\geq 0}$-Brownian motion defined on a probability space $(\Omega,\A,\p)$. The diffusion coefficient $\varepsilon\sigma(\cdot)$, with $0<\varepsilon<1$, is supposed to be known and to satisfy the following conditions:

(H1) $\sigma(\cdot)$ is a $K$-Lipschitz function on $\R$ bounded away from infinity and zero, i.e. there exist strictly positive constants $\sigma_0,\sigma_1, K$ with
\begin{equation}
\sigma_0^2\leq \sigma^2(y)\leq \sigma_1^2\ \textnormal{ and }\ |\sigma(z)-\sigma(y)|\leq K|z-y|,\quad \forall z,y\in \R.
\end{equation}
When $(y_t)$ is discretely observed we will also require the following assumption:

(H2)  $\sigma(\cdot)$ is a differentiable function on $\R$ with $K$-Lipschitz derivative, i.e.
\begin{equation*}
|\sigma'(z)-\sigma'(y)|\leq K|z-y|\quad \forall z,y\in \R.
\end{equation*}

More in details, we consider two experiments, the continuous one associated with $(y_t)$ and the discrete one given by the observations $(y_{t_1},\dots,y_{t_n})$, where $t_i=\frac{i}{n}T$.
Our aim is to prove that these nonparametric experiments are both equivalent to an autoregressive model given by Euler type discretization of $y$ with sampling interval $T/n,n\in \N^*$:
\begin{equation}\label{eq:Z}
Z_0=w,\quad Z_i=Z_{i-1}+\frac{T}{n}f(Z_{i-1})+\varepsilon\sqrt{\frac{T}{n}}\sigma(Z_{i-1})\xi_i, \quad i=1,\dots,n,
\end{equation}
with independent standard normal variables $\xi_i$.

The concept of asymptotic equivalence that we shall adopt is based on the Le Cam $\Delta$-distance between statistical experiments. Roughly speaking, saying that two statistical models, or experiments, are equivalent in the Le Cam sense means that any statistical inference procedure can be transferred from one model to the other in such a way that the asymptotic risk remains the same, at least for bounded loss functions. One can use this property in order to obtain asymptotic results working in a simpler but equivalent setting. For the basic concepts and a detailed description of the notion of asymptotic equivalence, we refer to \cite{lecam,LC2000}. A short review of this topic will be given in the Appendix. 

In recent years, the Le Cam theory on the asymptotic equivalence between statistical models has aroused great interest and a large number of works has been published on this subject. In parametric statistics, Le Cam's theory has successfully been applied to a huge variety of experiments. Proving an asymptotic equivalence for nonparametric experiments is more demanding but, nowadays, several works in this subject have appeared. The first results of global asymptotic equivalence for nonparametric experiments date from 1996 and are due to Brown and Low \cite{BL} and Nussbaum \cite{N96}. A non-exhausting list of subsequent works in this domain includes \cite{regression02,GN2002,ro04,C2007,cregression,R2008,C2009,R2013} for nonparametric regression, \cite{cmultinomial,j03,BC04} for nonparametric density estimation 
models, \cite{GN} for generalized linear models,  \cite{GN2006} for time series, \cite{B} for GARCH model, \cite{M2011} for functional linear regression, \cite{GN2010} for spectral density estimation and \cite{me} for inhomogeneous jumps diffusion models. Negative results are somewhat harder to come by; the most notable ones among them are \cite{sam96,B98,wang02}.

Asymptotic equivalence results have also been obtained for diffusion models. References concern nonparametric drift estimation with known diffusion coefficient. Among these one can quote \cite{D,CLN,R2006,rmultidimensionale,R11}. However, the most relevant results to our purposes are due to Milstein and Nussbaum \cite{NM} and to Genon-Catalot and Larédo \cite{C14}. The former authors have shown the asymptotic equivalence of a diffusion process continuously observed until time $T=1$ having unknown drift function and constant small known diffusion coefficient, with the corresponding Euler scheme. They also proved the asymptotic sufficiency of the discretized observation of the diffusion with small sampling interval. Hence, our work is a generalization of \cite{NM}. It can also be seen as a complement to \cite{C14}, the difference being that in our case the time horizon is kept fixed and the diffusion coefficient goes to zero. This setting allows for weaker hypotheses than those assumed by Genon-Catalot and 
Larédo (for example, we do not need the drift function $f$ to be uniformly bounded).

The interest in proving the asymptotic equivalence between the statistical model associated with the discretization of \eqref{eq:y} and \eqref{eq:Z} lies in the difficulty of making inferences in the discretely observed diffusion model. On the other hand, inference for model \eqref{eq:Z} is well understood and in practice one often reduces to working with the latter (see e.g. \cite{g90, l90,h99,co07}). The result in the present paper can thus be seen as a theoretical justification for such a practice.

The scheme of the proof is to prove both an asymptotic equivalence between the continuous and the discrete observation of \eqref{eq:y} and one between the continuous model \eqref{eq:y} and the Euler scheme \eqref{eq:Z}. By the triangular inequality, the result will follow. The main difficulty lies in the model \eqref{eq:Z} being equivalent to a diffusion process with a diffusion coefficient $\bar\sigma$ different from $\sigma$. In particular, this means that the total variation distance between \eqref{eq:Z} and \eqref{eq:y} is always 1. Thus, to prove the equivalence between these models it is necessary to construct an appropriate randomization. This is made possible by using random time changed experiments. Indeed, one can use random time changes in order to reduce to new diffusion models with diffusion coefficient equal to $\varepsilon$. However, these randomizations do not allow to apply the result of Milstein and Nussbaum directly since the changes of clock oblige to observe the new diffusion processes 
until different random times. Some care is then needed to overcome this technical obstacle (see Lemma \ref{lemma2}).


The paper is organized as follows: In Section 2 we give a brief presentation of the most relevant references connected with our work. Section 3 contains  the statement of the main results and a discussion while Section 4 is devoted to the proofs. The Appendix is devoted to background material.
\section{Existing literature}

As it has already been highlighted in the introduction, our result is not the first contribution in the context of asymptotic equivalences for diffusion processes. The aim of this section is to present the most relevant references linked with our work, that is \cite{NM}, \cite{C14} and \cite{R2006}. We recall below the results contained in these papers.  

\begin{itemize}
\item \emph{Diffusion approximation for nonparametric autoregression}, \cite{NM}: The authors consider the problem of estimating the function $f$ from a continuously or discretely observed diffusion process $y(t)$, $t\in[0,1]$, which satisfies the SDE
 $$dy_t=f(y_t)dt+\varepsilon dW_t,\quad t\in[0,1],\quad y_0=0,$$
 where $(W_t)$ is a standard Brownian motion and $\varepsilon$ is a known small parameter. 
 
The drift function $f(\cdot)$ is unknown and such that, for $K$ a positive constant, 
\begin{equation*}
 f\in\F_K=\Big\{f \textnormal{ defined on } \R \textnormal{ and }\forall  x,u\in\R,\ |f(x)-f(u)|\leq K|x-u|,\ |f(0)|\leq K\Big\}.
\end{equation*}
The constant $K$ has to exist but may be unknown. For what concerns the discrete observation of $(y_t)$ the authors place themselves in a high-frequency framework: $t_i=\frac{i}{n}$, $i\leq n$. Their main result is that, if $n\varepsilon\to\infty$ as $\varepsilon\to 0$, then there is an asymptotic equivalence between the continuous observation of $(y_t)$ and the corresponding Euler scheme:
$$Z_0=0,\quad Z_i=Z_{i-1}+\frac{f(Z_{i-1})}{n}+\frac{\varepsilon}{\sqrt n}\xi_i,\quad i=1,\dots,n,$$
where $(\xi_i)$ are i.i.d. standard normal variables.
Denoting by $\mo$ and $\mathscr{Z}_n$ the statistical models associated with the continuous observation of $(y_t)$ and the Euler scheme, respectively, an upper bound for the rate of convergence of $\Delta(\mo,\mathscr{Z}_n)$ is given by:
$$\Delta(\mo,\mathscr{Z}_n)\leq O\Big(\sqrt{n^{-2}\varepsilon^{-2}+n^{-1}}\Big), \quad \textnormal{as }\varepsilon\to 0.$$
The authors also prove that the discrete observations $(y_{t_1},\dots,y_{t_n})$ form an asymptotically sufficient statistics.
 \item \emph{Asymptotic equivalence of nonparametric diffusion and Euler scheme experiments}, \cite{C14}: 
 
The authors consider the diffusion process $(\xi_t)$ given by
\begin{equation}\label{eq:val}
d\xi_t=b(\xi_t)dt+\sigma(\xi_t)dW_t,\quad \xi_0=\eta, 
\end{equation}
where $(W_t)$ is a standard Brownian motion defined on a filtered probability space $(\Omega,\A,(\A_t)_{t\geq 0},\p)$, $\eta$ is a real valued random variable, $\A_0$-measurable, $b(\cdot)$, $\sigma(\cdot)$ are real-valued functions defined on $\R$. The diffusion coefficient $\sigma(\cdot)$ is a known nonconstant function that belongs to $C^2(\R)$ and satisfies the conditions:
\begin{equation*}
 \forall x\in \R,\ 0<\sigma_0^2\leq \sigma^2(x)\leq \sigma_1^2,\quad |\sigma'(x)|+|\sigma''(x)|\leq K_{\sigma}.
\end{equation*} The drift function $b(\cdot)$ is unknown and such that, for $K$ a positive constant, 
\begin{equation*}
 b(\cdot)\in\F_K=\Big\{b(\cdot)\in C^1(\R) \textnormal{ and for all } x\in\R,\ |b(x)|+|b'(x)|\leq K\Big\}.
\end{equation*}
The constant $K$ has to exist but may be unknown. The sample path of $(\xi_t)$ is continuously and discretely observed on a time interval $[0,T]$. The discrete observations of $(\xi_t)$ occur at the times $t_i=ih$, $i\leq n$ with $T=nh$. 
The authors prove the asymptotic equivalence between the continuous or discrete observation of $(\xi_t)$ and the corresponding Euler scheme:
$$Z_0=\eta,\quad Z_i=Z_{i-1}+hb(Z_{i-1})+\sqrt h\sigma(Z_{i-1})\varepsilon_i,$$
where, for $i\geq 1$, $t_i=ih$ and $\varepsilon_i=\frac{W_{t_i}-W_{t_{i-1}}}{\sqrt h}$.
The equivalences hold under the assumptions that $n$ tends to infinity with $h=h_n$ and $nh_n^2=\frac{T^2}{n}$ tending to zero. This includes both the case $T=nh_n$ bounded and the one $T\to\infty$. 
Let us stress the rate of convergence in the small variance case, that is obtained by replacing $\sigma$ by $\varepsilon\sigma$ in \eqref{eq:val}. Let us also denote by $\mo$ and $\mathscr{Z}_n$ the statistical models associated with the continuous observation of $(\xi_t)$ and the Euler scheme, respectively. The computations in \cite{C14} give the following upper bound for the rate of convergence:
\begin{equation*}
\Delta(\mo,\mathscr{Z}_n)\leq O\Big(\sqrt{n^{-2}\varepsilon^{-2}+n^{-1}+n^{-1}\varepsilon^{-4}}\Big).
\end{equation*}

\item \emph{Asymptotic statistical equivalence for scalar ergodic diffusions}, \cite{R2006}:
The authors focus on diffusions processes of the form
$$dX_t=b(X_t)dt+dW_t,\quad t\in[0,T].$$
For some fixed constants $C, A,\gamma>0$ the authors consider the non-parametric drift class
$$b\in\Sigma:=\Big\{b\in \text{Lip}_{\text{loc}}(\R): |b(x)|\leq C(1+|x|),\forall |x|>A: b(x)\frac{x}{|x|}\leq -\gamma\Big\},$$
where $\text{Lip}_{\text{loc}}(\R)$ denotes the set of locally Lipschitz continuous functions. The class $\Sigma$ has been chosen in order to ensure the existence of a stationary solution, unique in law, with invariant marginal density
$$\mu_b(x)=C_b\exp\Big(2\int_0^x b(y)dy\Big),\quad x\in\R,$$
where $C_b>0$ is a normalizing constant.
The main result of the paper is the asymptotic equivalence between the model associated with the continuous observation of $(X_t)$ and a certain Gaussian shift model, which can be interpreted as a regression model with random design.

However, in Section 4.5 the authors also consider discrete (high frequency) observations $(X_{t_i})_{i=0}^n$, where $0=t_0<t_1<\dots<t_n=T$, $d_i=t_i-t_{i-1}$ and $d_T=\max_{i=0,\dots,n-1}d_i$ goes to zero as $T$ goes to infinity. It is shown the asymptotic sufficiency of $(X_{t_0},\dots,X_{t_n})$ for $(X_t)_{t\in[0,T]}$ and the equivalence between the continuous observation of $X$ and its discrete counterpart, that is the autoregression model defined by observing $(y_1,\dots,y_n)$ from
$$y_{i+1}=y_i+d_i b(y_i)+\sqrt{d_i}\xi_i,\quad i=0,\dots,n-1,\quad y_0\sim \mu_b,$$
where the $\xi_i$'s are i.i.d. standard normal variables and independent of $y_0$.
\end{itemize}

\section{Main results}
To formulate our results we need to assume the standard conditions for existence and uniqueness of a strong solution $y$ for the SDE \eqref{eq:y} (\cite{ok}, Theorem 5.5, page 45). We shall thus work with parameter spaces included in $\F_M$, the set of all functions $f$ defined on $\R$ and satisfying

\begin{equation}\label{eq:ipf}
|f(0)|\leq M \textnormal{ and }|f(z)-f(y)|\leq M|z-y|,\quad \forall z,y\in \R.
\end{equation}
In particular, observe that every element of $\F_M$ satisfies a condition of linear growth: $|f(z)|\leq M(1+|z|)$, $\forall z\in\R$.
Let $C=C(\R^+,\R)$ be the space of continuous mappings $\omega$ from $\R^+$ into $\R$. Define the \emph{canonical process} $x:C\to C$ by 
$$\forall \omega\in C,\quad x_t(\omega)=\omega_t,\;\;\forall t\geq 0.$$
 Let $\Ci^0$ be the smallest $\sigma$-algebra of parts of $C$ that makes $x_s$, $s\geq0$, measurable. Further, for any 
 $t\geq 0$, let $\Ci_t^0$ be the smallest $\sigma$-algebra that makes $x_s$, $s$ in $[0,t]$, measurable. Finally, set $\Ci_t:=\bigcap_{s>t}\Ci_s^0$ and $\Ci:=\sigma\big(\Ci_t;t\geq 0\big)$. Let us denote by $P_f^{n,y}$ the distribution induced on $(C,\Ci_T)$ by the law of $y$, solution to \eqref{eq:y} and by $Q_f^{n,y}$ the distribution defined on $(\R^n,\B(\R^n))$ by the law of $(y_{t_1},\dots,y_{t_n})$, $t_i=T\frac{i}{n}$. We call $\mo_y^T$ the experiment associated with the continuous observation of $y$ until the time $T$ and $\mathscr{Q}_y^n$ the discrete one, based on the grid values of $y$:
 \begin{align}\label{eq:moy}
\mo_y^{T}&=\big(C,\Ci_T,\{P_f^{y},f\in \F\}\big),\\
\label{eq:moydiscreto}
\mathscr{Q}_y^n&=\big(\R^n,\B(\R^n),\{Q_f^{n,y},f\in \F\}\big).
\end{align}
Finally, let us consider the experiment associated with the Euler scheme corresponding to \eqref{eq:y}.
We denote by $Q_f^{n,Z}$ the distribution of $(Z_i, i=1,\dots,n)$ defined by \eqref{eq:Z}. Then:
\begin{equation}\label{eq:moz}
\Qi_Z^n=\big(\R^n,\B(\R^n),\{Q_f^{n,Z},f\in \F\}\big).
\end{equation}

Let us now state our main results.
\begin{theorem}\label{teo1}
Suppose that for some $M>0$ the parameter space $\F$ fulfills $\F\subset \F_M$ and that $\sigma(\cdot)$ satisfies Assumption (H1) with $K=M$. 
Then, if $\varepsilon n\to \infty$ as $n\to \infty$ and $\varepsilon\to 0$, the experiments $\mo_y^{T}$ and $\Qi_Z^n$ are asymptotically equivalent. More precisely we have 
$$\Delta\big(\mo_y^{T},\Qi_Z^n\big)=O\Big(\frac{1}{\varepsilon n}+(n^{-1}+\varepsilon)^{1/4}\Big).$$
\end{theorem}
\begin{theorem}\label{teo2}
Suppose that for some $M>0$ the parameter space $\F$ fulfills $\F\subset \F_M$ and that $\sigma(\cdot)$ satisfies assumptions (H1) and (H2), with $K=M$. Furthermore, require that $\sigma(\cdot)$ and $\F$ are such that $\frac{f}{\sigma}(\cdot)$ is $L$-Lipschitz with a uniform $L$ for $f\in\F$.
Then, for any (possibly fixed) $\varepsilon$, the sampled values $y_{t_1},\dots,y_{t_n}$ are an asymptotically sufficient statistic for the experiment $\mo_y^{T}$.
\end{theorem}
\begin{remark}
 If $\F\subset\F_M$ for some $M>0$ and $\sigma(\cdot)$ satisfies assumptions (H1) and (H2) then the Lipschitz condition on $\frac{f}{\sigma}(\cdot)$ is satisfied in both the following two cases: Either $|f(x)|\leq M$ for all $x$, or $|x\sigma'(x)|\leq M$ for all $x$.  
\end{remark}

\begin{corollary}
 Under the same hypotheses as in Theorem \ref{teo2}, the statistical model associated with the sampled values $y_{t_1},\dots,y_{t_n}$ is asymptotically equivalent to  $\Qi_Z^n$, as $n$ goes to infinity. The same upper bound for the rate of convergence that appears in Theorem \ref{teo1} holds.
\end{corollary}
\subsection{Discussion}\label{sec:discussion}
Our results are intended to be a generalization of \cite{NM} allowing to have a small but non-constant diffusion coefficient. 

The fact of being in a small variance case is crucial for the proof of our results. Indeed, as in \cite{C14}, we use auxiliary random time changed models to prove Theorem \ref{teo1} and a key step in the proof is the comparison between diffusion models having small (and constant) diffusion coefficient $\varepsilon$ observed until different stopping times. In particular, this leads to compare the $L_1$-norm between two stopping times (see Lemma \ref{lemma2} as opposed to Lemma 3.2 in \cite{C14}) and we take care of that by using the fact that any diffusion process with small variance converges to some deterministic solution. As a consequence, we find the same conditions as in \cite{NM} on $\varepsilon$ and $n$, i.e. $n\varepsilon\to \infty$, instead of conditions of the type $n\varepsilon^4\to \infty$ as in \cite{C14}. Another important feature of the small variance case is that it allows for weaker hypotheses than those assumed in \cite{C14}; more precisely, we only need to ask the same conditions about $\F$ 
as 
in \cite{NM}.     

Also remark that the most important novelty in the paper is Theorem \ref{teo1} since Theorem \ref{teo2} is a small generalization of the results already obtained in \cite{C14}, \cite{R2006}.

\section{Proofs}
In this section we collect the proofs of Theorems \ref{teo1} and \ref{teo2}. Since a lot of auxiliary statistical models are used to obtain our main results, we believe that starting by outlining the strategy of the proofs can be helpful to the reader. More precisely, seven models come into play in the proof of Theorem \ref{teo1}:
\begin{enumerate}[($\mathscr{M}_1$)]
 \item $dy_t=f(y_t)dt+\varepsilon\sigma(y_t)dW_t,\quad y_0=w,\quad t\in[0,T]$;
 \item $d\bar y_t=\bar f_n(t,\bar y)dt+\varepsilon\bar\sigma_n(t,\bar y)dW_t ,\quad \bar y_0=w ,\quad t\in[0,T]$;
 \item $d\xi_t=\frac{f}{\sigma^2}(\xi_t)dt+\varepsilon dW_t ,\quad \xi_0=w ,\quad t\in[0,A_T(\xi)]$;
 \item $d\bar\xi_t=\bar g_n(t,\bar\xi)dt+\varepsilon dW_t ,\quad \bar\xi_0=w ,\quad t\in[0,\bar A_T^n(\bar\xi)]$;
 \item $d\xi_t=\frac{f}{\sigma^2}(\xi_t)dt+\varepsilon dW_t,\quad \xi_0=w ,\quad t\in[0,S_T^n(\xi)]$;
 \item $d\xi_t=\frac{f}{\sigma^2}(\xi_t)dt+\varepsilon dW_t,\quad \xi_0=w,\quad t\in[0,\bar A _T(\bar\xi)]$;
\item $Z_0=w,\quad Z_i=Z_{i-1}+\frac{T}{n}f(Z_{i-1})+\varepsilon\sqrt{\frac{T}{n}}\sigma(Z_{i-1})\xi_i, \quad i=1,\dots,n$;
 \end{enumerate}
where $A_T(x)$, $\bar A_T^n(x)$ and $S_T^n(x)$ are certain $\Ci_t$-stopping times; $\bar f_n$, $\bar\sigma_n$, $\bar g_n$ are piecewise constant approximations of $f$ and $\sigma$, and the $\xi_i$'s are independent standard normal variables. 

The scheme of the proof is as follows:
  \begin{itemize}
  \item $\Delta(\M_1,\M_3)$, $\Delta(\M_2,\M_4)$ and $\Delta(\M_2,\M_7)$ are equal to zero: see, respectively, Propositions \ref{prop1}, \ref{prop:bar} and \ref{fisher}. 
   \item $\Delta(\M_5,\M_3)$ and $\Delta(\M_5,\M_6)$ are bounded by $\sqrt[4]{n^{-1}+\varepsilon}$ up to some constants: see Proposition \ref{prop3} 
   \item  $\Delta(\M_6,\M_4)=O\Big(\sqrt{n^{-1}+\varepsilon^{-2}n^{-2}}\Big)$: see Proposition \ref{prop2}.
 \end{itemize}
We thus deduce that the Le Cam $\Delta$-distance between our models of interest is bounded by:
$$\Delta(\M_1,\M_7)\leq O\Big(\sqrt[4]{n^{-1}+\varepsilon}+\varepsilon^{-1}n^{-1}\Big).$$
 On the other hand, six the models are used in the Proof of Theorem \ref{teo2}:
 \begin{enumerate}[($\mathscr{N}_1$)]
 \item $dy_t=f(y_t)dt+\varepsilon\sigma(y_t)dW_t,\quad y_0=w,\quad t\in[0,T]$;
  \item $(y_{t_1},\dots,y_{t_n})$;
  \item $d\mu_t=\Big(\frac{f(F^{-1}(\mu_t))}{\varepsilon\sigma(F^{-1}(\mu_t))}-\frac{\sigma'(F^{-1}(\mu_t))}{2\varepsilon}\Big)dt+dW_t,\quad \mu_0=F(w),\quad t\in[0,T]$;
  \item $(\mu_{t_1},\dots,\mu_{t_n})$;
  \item  $d\bar\mu_t=\bar b_n(t,\bar\mu)dt+dW_t, \quad \bar\mu_0=F(w),\quad t\in[0,T]$;
  \item $(\bar\mu_{t_1},\dots,\bar\mu_{t_n})$;
   \end{enumerate}
where $F(x)=\int_0^{x}\frac{1}{\varepsilon\sigma(u)}du$ and $\bar b_n$ is a piecewise constant approximation of a certain function $b$ depending on $f,\varepsilon,\sigma$ and $F$.

The strategy of the proof is:
\begin{itemize}
 \item $\Delta(\Nn_1,\Nn_3)$, $\Delta(\Nn_2, \Nn_4)$ and $\Delta(\Nn_5,\Nn_6)$ are equal to zero: see Propositions \ref{prop:mu} and \ref{prop:mud}; 
 \item $\Delta(\Nn_3, \Nn_5)$ and $\Delta(\Nn_6, \Nn_4)$ are bounded by $n^{-1}$ up to some constants: see Propositions \ref{prop:barmu} and \ref{prop:mud}, respectively.
 \end{itemize}
 
It follows that:
$$\Delta(\Nn_1, \Nn_2)=O\big(n^{-1}\big).$$

\subsection{Random time substitutions for Markov processes}
A key tool in establishing the asymptotic equivalence between the diffusion model continuously observed and its Euler scheme is given by random time changes for Markov processes. More in details we will need the following results.
\begin{theorem}\label{Vol}(see \cite{volk})
 Let $(Y,\p_y)$ be a (càdlàg) strong $(\A_t)$-Markov process on $(\Omega,\A,\p)$ with state space $(\R^d,\mathcal B(\R^d))$ and let $v: \R^d\to (0,\infty)$ be a positive continuous function. Define the additive functional 
 $$F_t=\int_0^t\frac{ds}{v(Y_s)},\ t\geq 0$$
 and assume that
 $$\int_0^{\infty}\frac{ds}{v(Y_s)}=\infty,\ \p_y-\textnormal{a.s. }, \forall y\in \R^d,$$
 so that the right continuous inverse 
 $$T_t=\inf\{s\geq 0: F_s>t\},\ t\geq 0$$
 of the functional $F$ is well defined on $[0,\infty).$ Then the process
 $$J_t=Y_{T_t},\ t\geq0,$$
 is a càdlàg strong $(\A_{T_t})$-Markov process on the probability space $(\Omega,\A,\p)$.
 
 Assume moreover that $(Y,\p_y)$ is a Feller process with infinitesimal generator $\mathcal L^Y$ and domain $\mathcal D$. Then $J$ is also a Feller process whose infinitesimal generator, with domain $\mathcal D$, is given by 
 \begin{equation*}
  \mathcal L^Jh(z)=v(z)\mathcal L^Yh(z),\quad h\in\mathcal D,\ z\in \R^d.
 \end{equation*}
\end{theorem} 
\begin{property}\label{proprieta1} 
For all $\omega\in C$,  $s,t>0$ define:
 \begin{align*}
\rho_s(\omega)=\int_0^s \sigma^2(\omega_r) dr;&\quad \eta_t(\omega)=\inf\big\{s\geq 0,\ \rho_s(\omega)\geq t\big\},\\
\theta_s(\omega)=\int_0^s \frac{1}{\sigma^2(\omega_r)}dr;&\quad  A_t(\omega)=\inf\big\{s\geq 0,\ \theta_s(\omega)\geq t\big\}.
\end{align*}
Then, the following hold:
\begin{enumerate}
 \item $\rho_T(x)= A_T(x_{\eta_{\cdot}(x)}),$
 \item $ A_t(x)=\int_0^t\sigma^2(x_{ A_s(x)})ds,\quad \forall t\in [0,T]$.
\end{enumerate}


\end{property}
\begin{proof}
 1. It is enough to show that $\theta_T(x_{\eta_{\cdot}(x)})=\eta_T^n(x)$ since $t\mapsto  A_t(x)$ and $t\mapsto\rho_t(x)$ are, respectively, the inverses of the applications $t \mapsto \theta_t(x)$ and $t\mapsto\eta_t(x)$.
 To prove the last assertion compute:
 \begin{align*}
  \theta_T(x_{\eta_{\cdot}(x)})=\int_0^T\frac{1}{\sigma^2(x_{\eta_r(x)})} dr=\int_0^{\eta_T(x)} \frac{\sigma^2(x_s)}{\sigma^2(x_s)}ds=\eta_T(x); 
 \end{align*}
where in the second equality we have performed the change of variable $s=\eta_r(x)\Leftrightarrow r=\rho_s(x)$ that yields $dr=\sigma^2(x_s)ds$. 

2. Again, we use that $t\mapsto \theta_t(x)$ is the inverse of the function $t\mapsto  A_t(x)$ combined with the following elementary fact:

Let $h$ and $g$ be two differentiable functions on $\R$ such that $h(0)=0=g(0)$ and their derivatives never vanish. Then, $h'(z)=\frac{1}{g'(h(z))}$ for all $z$ in $\R$ if and only if $h$ is the inverse of  $g$.

To show the assertion in 2. it is enough to apply this fact to $h(t)=A_t(x)$ and $g(t)=\theta_t(x)$.
 \end{proof}

\subsection{Proof of Theorem \ref{teo1}}
 We will proceed in four steps. More precisely, in Step 1 we consider a random time change on the diffusion \eqref{eq:y} in order to obtain an experiment equivalent to $\mo_y^{T}$ but associated with a diffusion having diffusion coefficient equal to $\varepsilon$. In Step 2 we construct a continuous time discretization of the process $(y_t)$ and, applying a second random time change, we prove an equivalence result between a second experiment associated again with a diffusion having diffusion coefficient equal to $\varepsilon$. In Step 3 we compare, in term of the Le Cam $\Delta$-distance, the two experiments with the diffusion coefficient equal to $\varepsilon$ constructed in Steps 1-2. Finally, in Step 4, we prove the equivalence between the experiment associated with the continuous time discretization of $(y_t)$ and the one with the Euler scheme. By means of the triangular inequality we are able to bound the Le Cam $\Delta$-distance between $\mo_y^{T}$ and $\Qi_Z^n$.
 
 \textbf{Step 1.} We start by proving the Le Cam equivalence between $\mo_y^{T}$ and a corresponding diffusion model with coefficient diffusion equal to $\varepsilon$.
 Recall that $P_f^{y}$ is the law on $(C,\Ci_T)$ of a diffusion process with infinitesimal generator $\mathcal L_1$ given by
 \begin{equation}\label{geny}
  \mathcal L_1= f\nabla+\varepsilon^2\frac{\sigma^2}{2}\Delta,
  \end{equation}
  Define $P_f^{\xi}$ as the law on $(C,\Ci)$ of a diffusion process with infinitesimal generator $\mathcal L_2$ given by
 \begin{equation}\label{genxi}
  \mathcal L_2= \frac{f}{\sigma^2}\nabla+\varepsilon^2\frac{1}{2}\Delta
 \end{equation}
 and initial condition $\xi_0=\omega.$
Moreover, for all $A>0$ a $\Ci_t$-stopping time, define the experiment
$$\mo_{\xi}^{A}=(C,\Ci_{A},(P_f^{\xi}|_{\Ci_{A}}, f\in \F)).$$
\begin{prop}\label{prop1}
 Suppose that for some $M>0$ the parameter space $\F$ fulfills $\F\subset \F_M$ and that $\sigma(\cdot)$ satisfies assumption (H1), with $K=M$. Then, the statistical models $\mo_{y}^{T}$ and $\mo_{\xi}^{A_T(x)}$ are equivalent.
\end{prop}
\begin{proof}
 Let us prove that $\delta(\mo_{y}^{T}, \mo_{\xi}^{A_T(x)})=0$. Note that $(x_t)$ under $P_f^{y}$ is a $(\Ci_t)$-Markov process with infinitesimal generator as in \eqref{geny}. Define a new process $\xi$ as a change of time of $(x_t)$ with stochastic clock $(\eta_t(x))_t$: $\xi_0=w,$ $\xi_t:=x_{\eta_t(x)},$ $\forall t>0$. Theorem \ref{Vol} ensures that the process $(\xi_t)_{t\geq 0}$ is a diffusion process with infinitesimal generator given by \eqref{genxi}.
 Also, remark that, as $(x_t)$ is defined on $[0,T]$, then the trajectories of $(\xi_t)$ are defined until the time $\rho_T(x)$. In order to produce a  randomization transforming the family of measures $\{P_f^{y}, \ f\in \F\}$ in $\{P_f^{\xi}|\Ci_{ A_T(x)}, \ f\in \F\}$, let us consider the following application:
 $$
 \Phi:\{\omega_t:\ t\in [0,T]\}\to \{\omega_{\eta_t(\omega)}:\ t\in [0,\rho_T(\omega)]\}.
 $$
 Observe that the process $\Phi(x)$ is defined until the time $\rho_T(x)$ that is equal to $ A_T(\Phi(x))$ (see Property \ref{proprieta1}), so that any set of paths of $\Phi(x)$ belongs to $\Ci_{ A_T(x)}$. 
Introduce the Markov kernel $K$ defined by $K(\omega,\Gamma)=\I_{\Gamma}(\Phi(\omega))$, $\forall \omega\in C$, $\forall \Gamma\in \Ci_{ A_T(x)}$,  then:
$$KP_f^{y}(\Gamma)=\int \I_{\Gamma}(\Phi(\omega))P_f^{y}(d\omega)=P_f^{y}(\Phi(x)\in\Gamma)=P_f^{\xi}|_{\Ci_{ A_T(x)}}(\Gamma).$$
Therefore $\delta(\mo_{y}^{T}, \mo_{\xi}^{A_T(x)})=0$.

The same type of computations imply that $\delta(\mo_{\xi}^{A_T(x)},\mo_y^{T})=0$ through use of the application $\Psi: (\omega_t:\ t\in[0, A_T(\omega)]) \to (\omega_{ A_t(\omega)}:\ t\in[0,T]) $.
\end{proof}
\textbf{Step 2.} We now introduce a statistical model that approximates  the model $\mo_y^{T}$. 
Given a path $\omega$ in $C$ and a time grid $t_i=T \frac{i}{n}$, we define
$$\bar{f}_n(t,\omega)=\sum_{i=1}^{n-1}f\big(\omega(t_{i})\big)\I_{[t_{i},t_{i+1})}(t), \quad \bar{\sigma}_n(t,\omega)=\sum_{i=1}^{n-1}\sigma\big(\omega(t_{i})\big)\I_{[t_{i},t_{i+1})}(t), \quad \forall t\in[0,T].$$
Then, we denote by $P_f^{n,\bar y}$ the law on $(C,\Ci_T)$ of a diffusion process with infinitesimal generator $\bar{\mathcal L}^n$ given by
 \begin{equation}\label{genbary}
 \bar{\mathcal L}^n_t(\omega)h(z)= \bar f_n(t,\omega)\nabla h(z)+\varepsilon^2\frac{\bar\sigma_n^2(t,\omega)}{2}\Delta h(z),\quad \forall \omega \in C,\ h \in C^2(\R),\ z \in \R
 \end{equation} 
 and initial condition $\bar y_0=\omega.$
 Consider the experiment
 $$\mo_{\bar y}^{n,T}=\big(C,\Ci_T,(P_f^{n,\bar y}|_{\Ci_T}, f\in \F)\big)$$
 Again, we want to introduce the diffusion model with diffusion coefficient equal to $\varepsilon$ associated to $\mo_{\bar y}^{n,T}$. To that aim, for all $\omega\in C$, define
 \begin{align}
\bar A_0^n(\omega)&=0, \quad \bar A^n_t(\omega)=\bar A_{t_{i-1}}(\omega)+\sigma^2(\omega_{\bar A_{t_{i-1}}}(\omega))(t-t_{i-1}), \quad t\in(t_{i-1},t_i];\label{eq:A}\\
\bar g_n(t,\omega)&=\sum_{i=1}^n\frac{f}{\sigma^2}(\omega_{\bar  A^n_{t_i}(\omega)})\I_{\big(\bar  A^n_{t_i}(\omega),\bar  A^n_{t_{i+1}}(\omega)\big]}(t),\quad t\geq 0,\quad i=0,\dots,n-1 \label{eq:g_n}.
 \end{align}
  \begin{lemma}\label{lemma:barA}
  $\bar A_{t_i}^n(x)$ is a $\Ci_t$-stopping time for all $i=1,\dots,n$.
 \end{lemma}
\begin{proof}
By \eqref{eq:A}, $\bar A_{t_1}^n(x)=\sigma^2(x(0))t_1$, so the set
$\{\bar A_{t_1}^n(x)\leq t\}=\begin{cases} \emptyset &\mbox{if } \sigma^2(x(0))t_1>t\\
                                             C        &\mbox{otherwise.}                                                                            
                                                                                                                               
                                                                                                            \end{cases}$
belongs to $\Ci_t$, for all $t$. By induction, assume that $\bar A_{t_{i-1}}^n(x)$ is a $(\Ci_t)$-stopping time and remark that \eqref{eq:A} implies $\{\bar A_{t_i}^n(x)\leq t\}=\{\bar A_{t_i}^n(x)\leq t\}\cap\{\bar A_{t_{i-1}}^n(x)\leq t\}$.  Since $(x_t)$ is $(\Ci_t)$-adapted and continuous, in particular it is progressively measurable with respect to $(\Ci_t)$. By the induction hypothesis it follows that $x_{\bar A_{t_{i-1}}^n(x)}$ is $\Ci_{\bar A_{t_{i-1}}(x)}$- measurable, hence, using \eqref{eq:A}, $\{\bar{A}^n_{t_i}(x)\leq t\}\in \Ci_{\bar{A}^n_{t_{i-1}}(x)}$, as $\bar{A}^n_{t_{i-1}}(x)$ is already $\Ci_{\bar{A}^n_{t_{i-1}}(x)}$-measurable, again by the induction hypothesis.
By the definition of the $\sigma$-algebra $\Ci_{\bar A_{t_{i-1}}(x)}$ and the induction hypothesis, we then conclude that $\{\bar A_{t_i}^n(x)\leq t\} \in\Ci_t$. Hence the result.
\end{proof}
Denote by $P_f^{n,\bar\xi}$ the law on $(C,\Ci)$ of a diffusion process with infinitesimal generator $\tilde{\mathcal L}^n$ given by
 \begin{equation}\label{genbarxi}
 \tilde{\mathcal L}^n_t(\omega)h(z)=\bar g_n(t,\omega)\nabla h(z)+\frac{\varepsilon^2}{2}\Delta h(z), \quad \forall \omega \in C,\ h \in C^2(\R),\ z \in \R
 \end{equation}
 and initial condition $\bar\xi_0=\omega$.
Thanks to Lemma \ref{lemma:barA} we can define the statistical model associated with the observation of $\bar\xi$ until the stopping time $\bar A_T^n(x)$:
\begin{equation*}
 \mo_{\bar\xi}^{n,\bar A_T^n(x)}=\big(C,\Ci_{\bar A_T^n(x)},\{P_f^{n,\bar\xi}|_{\Ci_{\bar A_T^n(x)}},f\in \F\}\big).
\end{equation*}
As in Step 1, one can prove the following proposition. There are, however, some technical points that need to be taken care of; for more details, we refer to \cite{C14}, Proposition 5.4.
\begin{prop}\label{prop:bar}
 Under the same hypotheses as in Proposition \ref{prop1}, the statistical models $\mo_{\bar y}^{n,T}$ and $\mo_{\bar\xi}^{n,\bar A_T^n(x)}$ are equivalent.
\end{prop}
\textbf{Step 3.} We shall prove that $\Delta(\mo_{\xi}^{A_T(x)},\mo_{\bar\xi}^{n,\bar A_T^n(x)})\to 0$ as $n\to\infty$. To that aim we will prove that, setting $S_T^n(x):= A_T(x)\wedge \bar  A_T^n(x)$, $\Delta(\mo_{\xi}^{A_T(x)},\mo_{\xi}^{S_T^n(x)})\to 0$, $\Delta(\mo_{\xi}^{S_T^n(x)},\mo_{\xi}^{\bar A_T^n(x)})\to 0$ and $\Delta(\mo_{\xi}^{\bar A_T^n(x)},\mo_{\bar\xi}^{n,\bar A_T^n(x)})\to 0$ as $n\to\infty$. We shall start by showing that $\Delta(\mo_{\xi}^{\bar A_T^n(x)},\mo_{\bar\xi}^{n,\bar A_T^n(x)})\to 0$; we need the following lemmas:
\begin{lemma}\label{lemmazeta}
The law of $(x_{A_t(x)})$ under $P_f^{\xi}$ is the same as the law of $(x_t)$ under $P_f^{y}$.
Moreover, let $P_f^{\bar{\zeta}}$ be the distribution induced on $(C,\Ci)$ by the law of a diffusion process $(\bar{\zeta}_t)$ satisfying 
 \begin{equation*}
  d\bar{\zeta}_t=\frac{f(\bar{\zeta}_t)}{\sigma^2(\bar{\zeta}_t)}\bar{\sigma}_n^2(t,\bar{\zeta})dt+\varepsilon dW_t,\quad \bar{\zeta}_0=w.
 \end{equation*}
Then, the law of $(x_{\bar A^n_t(x)})$ under $P_f^{\xi}$ is the same as the law of $(x_t)$ under $P_f^{\bar{\zeta}}$.
\end{lemma}
\begin{proof}
We shall only prove the first assertion, the proof of the second one being very similar.
 We have $\xi_0=w=y_0$ and, for all $t>0$:
 $$x_{A_t(x)}=w+\int_0^{A_t(x)}\frac{f(x_s)}{\sigma^2(x_s)}ds+\varepsilon \tilde{W}_{A_t^n(x)},$$
 where the process $(\tilde W_t)$ is a standard Brownian motion under $P_f^{\xi}$.
 The change of variable $s= A_u(x)$ implies that $ds=\sigma^2(x_{A_u(x)})du$, hence one can write
 $$x_{A_t(x)}=w+\int_0^t\frac{f(x_{A_u(x)})}{\sigma^2(x_{A_u(x)})}\sigma^2(x_{A_u(x)})du +\varepsilon \int_0^t\sigma(x_{A_{u}(x)})B_u,$$
 where the process $(B_t)$ is defined by
 $$B_t=\int_0^t \frac{d\tilde W_{A_u(x)}}{\sigma(x_{A_{u}(x)})}.$$
 Classical results (see e.g. \cite{kara}, 5.5) ensure that $(B_t)$ is a $\Ci_{A_t(x)}$ standard Brownian motion under $P_f^{y}$. It follows that the law of $x_{A_t(x)}$ under $P_f^{\xi}$ is the same as the law of $x$ under $P_f^{y}$. 
\end{proof}

\begin{lemma}\label{Gro}
Let $p$ be an even positive integer and $(t_n)$ a sequence of times bounded by $CT$ for some constant $C$ independent of $f$; then 
$\E_{P_f^{\bar{\zeta}}}|x_{t_n}|^p=O(1)$, uniformly on $\F$.
\end{lemma}
\begin{proof}
In order to bound $\E_{P_f^{\bar{\zeta}}}|x_{t_n}|^p$ we will use the following facts:
\begin{itemize}
 \item $(z_1+\dots+z_m)^p\leq m^{p-1}(z_1^p+\dots+z_m^p)$, $\forall z_1,\dots,z_m\in \R$;
 \item $\E_{P_f^{\bar{\zeta}}}|\int_0^vh(x_r)dr|^p\leq v^{p-1}\int_0^v\E_{P_f^{\bar{\zeta}}}h^p(x_r)dr$, for any integrable function $h$;
\item If $X$ is a centered Gaussian random variable with variance $\sigma^2$ then $\E[X^p]=\sigma^{p}(p-1)!!$;
 \item (Gronwall lemma) Let $I=[0,a]$ be an interval of the real line, $\alpha$ a constant and let $\beta$ and $u$ continuous real valued functions defined on $I$. If $\beta$ is non-negative and if $u$ satisfies the integral inequality:
 $$u(t) \leq \alpha + \int_0^t \beta(s) u(s)\,\mathrm{d}s,\qquad \forall t\in I,$$
 then
 $$u(t) \leq \alpha\exp\bigg(\int_0^t\beta(s)\,\mathrm{d}s\bigg),\qquad t\in I.$$
 \end{itemize}
 As one can always construct a Brownian motion $(\bar B_t)$ under $P_f^{\bar{\zeta}}$ such that $dx_t=\frac{f(x_t)}{\sigma^2(x_t)}\bar{\sigma}_n^2(t,x)dt+\varepsilon d\bar B_t$, applying the first three facts combined with the linear growth of $f$ one can write:
 \begin{align*}
  \E_{P_f^{\bar{\zeta}}}&|x_{t_n}|^p \leq3^{p-1}w^p+3^{p-1}\E_{P_f^{\bar{\zeta}}}\bigg(\int_0^{t_n}\frac{f(x_s)}{\sigma^2(x_s)}\bar{\sigma}_n^2(s,x)ds\bigg)^p+
3^{p-1}\varepsilon^p\E_{P_f^{\bar{\zeta}}}\bar B_{t_n}^p\\
 &\leq 3^{p-1}w^p+3^{p-1}\frac{\sigma_1^{2p}}{\sigma_0^{2p}}(CT)^{p-1}\int_0^{t_n}\E_{P_f^{\bar{\zeta}}}[f^p(x_s)]ds+
3^{p-1}\varepsilon^p(CT)^{\frac{p}{2}}(p-1)!!\\
 &\leq C'\bigg(1+\int_0^{t_n}\E_{P_f^{\bar{\zeta}}}|x_s|^pds\bigg),
\end{align*}
for some constant $C'$ independent of $f$. 
Applying the Gronwall lemma, we obtain 
$$\E_{P_f^{\bar{\zeta}}}|x_{t_n}|^p\leq C'e^{C't_n}\leq C'e^{C'CT}=O(1).$$
\end{proof}



\begin{lemma}\label{lemma1} Under the same hypotheses as in Proposition \ref{prop1} and with the same notation as in Steps 1 and 2, we have
 $$\E_{P_f^{\xi}}\int_0^{\bar A_T^n(x)}\Big(\frac{f(x_s)}{\sigma^2(x_s)}-\bar g_n(s,x)\Big)^2ds=O\big(n^{-2}+\varepsilon n^{-1}\big),$$
 uniformly on $\F$.
\end{lemma}
\begin{proof}
For the sake of brevity, in this proof we will omit the superscript $n$ in each occurrence of $\bar A_t^n$.
We start by observing that, for all $y,z\in\R$
\begin{align*}
 \bigg|\frac{f(z)}{\sigma^2(z)}-\frac{f(y)}{\sigma^2(y)}\bigg|&\leq|f(z)|\bigg|\frac{1}{\sigma^2(z)}-\frac{1}{\sigma^2(y)}\bigg|+\frac{|f(z)-f(y)|}{\sigma_0^2}\\
 &\leq\frac{2M^2\sigma_1}{\sigma_0^4}(1+|z|)|z-y|+\frac{M}{\sigma_0^2}|z-y|,
\end{align*}
hence there exists some constant $C$ such that $\Big(\frac{f(z)}{\sigma^2(z)}-\frac{f(y)}{\sigma^2(y)}\Big)^2\leq C(z-y)^2(1+z^2)$.

Applying this inequality we can write:
\begin{align*}
  \int_0^{\bar  A_T(x)}\Big(\frac{f(x_s)}{\sigma^2(x_s)}-\bar g_n(s,x)\Big)^2ds&\leq \sum_{i=0}^{n-1}\int_{\bar A_{t_i}(x)}^{\bar A_{t_{i+1}}(x)}C\big(x_s-x_{\bar A_{t_i}(x)}\big)^2(1+x^2_{\bar A_{t_i}(x)})ds\\
  &=C\sum_{i=0}^{n-1}(1+x^2_{\bar A_{t_i}(x)})\int_{\bar A_{t_i}(x)}^{\bar A_{t_{i+1}}(x)}(x_r-x_{\bar A_{t_i}(x)})^2dr\\
  &\leq C\sigma_1^2\sum_{i=0}^{n-1}(1+x^2_{\bar A_{t_i}(x)})\int_0^{t_{i+1}-t_i}(x_{\bar A_{t_i+s}(x)}-x_{\bar A_{t_i}(x)})^2ds,
 \end{align*}
where in the last step we have performed the change of variables $r=\bar A_{t_i+s}(x)$.
 Thanks to the Cauchy-Schwarz inequality and Lemma \ref{Gro} we obtain
 \begin{align*}
  \E_{P_f^{\xi}}&\int_0^{\bar  A_T(x)}\Big(\frac{f(x_s)}{\sigma^2(x_s)}-\bar g_n(s,x)\Big)^2ds\leq \\
  &\leq C\sigma_1^2\sum_{i=0}^{n-1}\sqrt{\E_{P_f^{\xi}}(1+x^2_{\bar A_{t_i}(x)})^2}\sqrt{\E_{P_f^{\xi}}\bigg(\int_0^{t_{i+1}-t_i}(x_{\bar A_{t_i+r}(x)}-x_{\bar A_{t_i}(x)})^2dr\bigg)^2}\\
  &=C\sigma_1^2\sum_{i=0}^{n-1}\sqrt{\E_{P_f^{\bar{\zeta}}}(1+x^2_{t_i})^2}\sqrt{\E_{P_f^{\bar{\zeta}}}\bigg(\int_0^{\frac{T}{n}}(x_{r+t_i}-x_{t_i})^2dr\bigg)^2}\\
   &\leq  C\sigma_1^2 \sum_{i=0}^{n-1}\sqrt{\E_{P_f^{\bar{\zeta}}}(2+2x^4_{t_i})}\sqrt{\E_{P_f^{\bar{\zeta}}}\bigg(\frac{T}{n}\int_0^{\frac{T}{n}}(x_{r+t_i}-x_{t_i})^4dr\bigg)}\\
   &=O\Bigg(\sum_{i=0}^{n-1}\sqrt{\E_{P_f^{\bar{\zeta}}}\Big(\frac{T}{n}\int_0^{\frac{T}{n}}(x_{r+t_i}-x_{t_i})^4dr\Big)}\Bigg).
 \end{align*}
Using the same arguments as in the proof of Lemma \ref{Gro}, we can write
 \begin{align*}
  \E_{P_f^{\bar{\zeta}}}(x_{r+t_i}-x_{t_i})^4&=\E_{P_f^{\bar{\zeta}}}\bigg|\int_{t_i}^{t_i+r}\frac{f(x_s)\sigma^2(x_{t_i})}{\sigma^2(x_s)}ds+\varepsilon\sigma(x_{t_i})(\bar B_{t_i+r}-\bar B_{t_i})\bigg|^4\\
                                   &\leq 8\E_{P_f^{\bar{\zeta}}} \bigg|\int_{t_i}^{t_i+r}\frac{f(x_s)\sigma^2(x_{t_i})}{\sigma^2(x_s)}ds\bigg|^4+8\varepsilon^4\sigma_1^4\E_{P_f^{\bar{\zeta}}} \big|\bar B_{t_i+r}-\bar B_{t_i}\big|^4\\
                                   &\leq 8 \frac{\sigma_1^8}{\sigma_0^8}r^3\int_0^{r}\E_{P_f^{\bar{\zeta}}}\big[f^4(x_{s+t_i})\big]ds+8\varepsilon^4 \sigma_1^4r^{2}6!\\
                                 &=O\bigg(r^4+r^{3}\int_0^{r}\E_{P_f^{\bar{\zeta}}}[x^4_{s+t_i}]ds+\varepsilon^4\ r^2\bigg)\\
                                 &=O(r^4+\varepsilon^4 r^2)=O\Big(\frac{1}{n^4}+\frac{\varepsilon^4}{n^2}\Big).
 \end{align*}
 Putting all the pieces together we get:
\begin{align*}
 \int_0^{\bar  A_T(x)}\Big(\frac{f(x_s)}{\sigma^2(x_s)}-\bar g_n(s,x)\Big)^2ds&=O\bigg(\sum_{i=0}^{n-1}\sqrt{\Big(\frac{T}{n}\int_0^{\frac{T}{n}}O\Big(\frac{1}{n^4}+\frac{\varepsilon^4}{n^2}\Big)dr\Big)}\bigg)\\&=O\bigg(\frac{1}{n^2}+\frac{\varepsilon^2}{n}\bigg).
\end{align*}

\end{proof}
\begin{prop}\label{prop2}
 Under the same hypotheses as in Proposition \ref{prop1}, we have $$\Delta(\mo_{\xi}^{\bar A_T^n(x)},\mo_{\bar\xi}^{n,\bar A_T^n(x)})=O(\sqrt{\varepsilon^{-2}n^{-2}+n^{-1}}).$$
\end{prop}
\begin{proof}
 We use an inequality involving the Hellinger process in order to bound $\Big\|P_f^{\xi}|\Ci_{\bar  A_T^n(x)}-P_f^{n,\bar\xi}|\Ci_{\bar  A_T^n(x)}\Big\|_{\textnormal{TV}}$ and hence $\Delta(\mo_{\xi}^{\bar A_T^n(x)},\mo_{\bar\xi}^{n,\bar A_T^n(x)})$. More precisely, let $h_f$ be the Hellinger process of order $1/2$ 
 between the measures $P_f^{\xi}|\Ci_{\bar  A_T^n(x)}$ and $P_f^{n,\bar\xi}|\Ci_{\bar  A_T^n(x)}$, that is, (see Jacod and Shiryaev,  \cite{jacod}, page 239)
 $$h_f(t)(x)=\frac{1}{8\varepsilon^2}\int_0^t \bigg(\frac{f(x_s)}{\sigma^2(x_s)}-\bar g_n(s,x)\bigg)^2 ds.$$
 Then:
 \begin{equation*}
  \Big\|P_f^{\xi}|_{\Ci_{\bar  A_T^n(x)}}-P_f^{n,\bar\xi}|_{\Ci_{\bar  A_T^n(x)}}\Big\|_{\textnormal{TV}}\leq 4 \sqrt{\E_{P_f^{\xi}} h_f(\bar  A_T^n(x))(x)},
 \end{equation*}
as in \cite{jacod}, 4b, Theorem 4.21, page 279. 
Hence we conclude thanks to Lemma \ref{lemma1}.
\end{proof}
We now prove that $\Delta(\mo_{\xi}^{A_T^n(x)},\mo_{\xi}^{S_T^n(x)})\to 0$ and $\Delta(\mo_{\xi}^{S_T^n(x)},\mo_{\xi}^{\bar A_T^n(x)})\to 0$ as $n\to\infty$. Again, we start with a lemma:
\begin{lemma}\label{lemma2}
 Under the hypotheses of Proposition \ref{prop1} and with the same notation as in Steps 1 and 2, we have
 \begin{equation}\label{eq:diffta}
  \E_{P_f^{\xi}}| A_T(x)-\bar  A_T^n(x)|=O\Big(\frac{1}{n}+\varepsilon\Big),
 \end{equation}
 uniformly over $\F$.
 \end{lemma}
\begin{proof}
The crucial point in proving \eqref{eq:diffta} is to use the convergence of diffusion processes with small variance to some deterministic solution. To that aim, let us introduce the following ODEs:
$$\frac{dz_t}{dt}=f(z_t),\quad \frac{d\bar z_t}{dt}=\frac{f(\bar z_t)}{\sigma^2(\bar z_t)}\bar{\sigma}_n^2(t,\bar{z}),\quad z_0=w=\bar z_0.$$
By means of Property \ref{proprieta1}, the Lipschitz character of $\sigma^2(\cdot)$ and the linear growth of $f$, we get,

\begin{align*}
 \E_{P_f^{\xi}}|A_T(x)-\bar A_T^n(x)|&=\E_{P_f^{\xi}}\bigg|\int_0^T\big(\sigma^2(x_{A_t(x)})-\bar{\sigma}_n^2(t,x_{\bar A_{\cdot}^n (x)})\big)dt\bigg|\\
                  &\leq 2\sigma_1M\E_{P_f^{\xi}}\sum_{i=0}^{n-1}\int_{t_i}^{t_{i+1}}|x_{A_t(x)}-x_{\bar A_{t_i}^n(x)}|dt\\
                  & \leq 2\sigma_1M\E_{P_f^{\xi}}\sum_{i=0}^{n-1}\int_{t_i}^{t_{i+1}}\big(|x_{A_t(x)}-z_t|+|z_t-\bar{z}_{t_i}|+|\bar z_{t_i}-x_{\bar A_{t_i}^n(x)}|\big)dt.
\end{align*}
For all $t\in[t_i,t_{i+1}]$, we shall analyze the terms $I=\E_{P_f^{\xi}}|x_{A_t(x)}-z_t|$, $II=|z_t-\bar{z}_{t_i}|$ and $III=\E_{P_f^{\xi}}|\bar{z}_{t_i}-x_{\bar{A}^n_{t_i}(x)}|$, separately.
\begin{itemize}
 \item Term I: By means of Lemma \ref{lemmazeta} and some standard calculations one can write
 \begin{align*}
  \E_{P_f^{\xi}}|x_{A_t(x)}-z_t|&=\E_{P_f^{y}}|x_{t}-z_t|=\E_{P_f^{y}}\bigg|\int_0^t\big(f(x_s)-f(z_s)\big)ds+\varepsilon\int_0^t\sigma(x_s)dW_s\bigg|\\
                               &\leq M \E_{P_f^{y}}\int_0^t|x_s-z_s|ds+\varepsilon\sigma_1\sqrt t;
 \end{align*}
hence, an application of the Gronwall lemma yields $\E_{P_f^{\xi}}|x_{A_t(x)}-z_t|\leq \varepsilon\sqrt t \exp(MT).$
\item Term II: By the triangular inequality it is enough to bound $|z_t-z_{t_i}|$ and $|z_{t_i}-\bar z_{t_i}|$, separately. It is easy to see that $|z_s- z_{t_j}|$ is a $O(n^{-1})$ as well as $|\bar z_s-\bar z_{t_j}|$ for all $s\in[t_j,t_{j+1}]$, $j=0,\dots,n-1$. Moreover, observe that there exists a constant $C$, independent of $f$, such that $\Big|f(x)-\frac{f(y)}{\sigma^2(y)}\sigma^2(z)\Big|\leq C(|x-y|+(1+|y|)|y-z|)$. We get:
\begin{align*}
|z_{t_i}-\bar z_{t_i}|&=\bigg|\int_0^{t_i}\Big(f(z_s)-\frac{f(\bar z_s)}{\sigma^2(\bar z_s)}\bar{\sigma}_n^2(s,\bar z)\Big)ds\bigg|\\
&=\bigg|\sum_{j=0}^{i-1}\int_{t_j}^{t_{j+1}}\Big(f(z_s)-\frac{f(\bar z_s)}{\sigma^2(\bar z_s)}\sigma^2(\bar z_{t_j})\Big)ds\bigg|\\
&\leq C \sum_{j=0}^{i-1}\int_{t_j}^{t_{j+1}}\big(|z_s-\bar z_s|+(1+|\bar z_s|)|\bar z_s-\bar z_{t_j}|\big)ds\\
&\leq C\int_0^{t_i}|z_s-\bar z_s|ds+\frac{C't_i}{n},
\end{align*}
for some constant $C'$, independent of $f$. Therefore, applying the Gronwall lemma one obtains
$$|z_{t_i}-\bar z_{t_i}|\leq \frac{C't_i}{n}e^{Ct_i}$$
that allows us to conclude $|z_t-\bar{z}_{t_i}|=O(n^{-1})$.
\item Term III: By means of Lemma \ref{lemmazeta} we know that $\E_{P_f^{\xi}}|\bar{z}_{t_i}-x_{\bar{A}^n_{t_i}(x)}|=\E_{P_f^{\bar{\zeta}}}|\bar{z}_{t_i}-x_{t_i}|$. 

\begin{align*}
  \E_{P_f^{\bar{\zeta}}}|\bar z_{t_i}-x_{t_i}|&=\E_{P_f^{\bar{\zeta}}}\bigg|\sum_{j=0}^{i-1}\int_{t_j}^{t_{j+1}}\Big[\Big(\frac{f(\bar z_s)}{\sigma^2(\bar z_s)}\sigma^2(\bar z_{t_j})-\frac{f(x_s)}{\sigma^2(x_s)}\sigma^2(x_{t_j})\Big)ds+\varepsilon \sigma(x_{t_j})dW_s\Big]\bigg|\\
  &\leq\E_{P_f^{\bar{\zeta}}}\bigg|\sum_{j=0}^{i-1}\int_{t_j}^{t_{j+1}}\Big(\frac{f(\bar z_s)}{\sigma^2(\bar z_s)}\sigma^2(\bar z_{t_j})-\frac{f(x_s)}{\sigma^2(x_s)}\sigma^2(x_{t_j})\Big)ds\bigg|+\varepsilon \sigma_1\sqrt{t_i}\\
  &\leq \E_{P_f^{\bar{\zeta}}}\sum_{j=0}^{i-1}\int_{t_j}^{t_{j+1}}\bigg(\frac{M\sigma_1^2}{\sigma_0^2}|\bar z_s-x_s|+\frac{2\sigma_1M}{\sigma_0^4}(1+|\bar z_s|)(|\bar z_{t_j}-\bar z_s|+|x_s-x_{t_j}|)\bigg)ds+\varepsilon \sigma_1\sqrt{t_i}\\
  &\leq \E_{P_f^{\bar{\zeta}}}\int_{0}^{t_i}\frac{M\sigma_1^2}{\sigma_0^2}|\bar z_s-x_s|ds+Cn^{-1}t_i+\varepsilon \sigma_1\sqrt{t_i},
 \end{align*}
for some constant $C$ independent of $f$. An application of the Gronwall lemma gives
$$\E_{P_f^{\bar{\zeta}}}|\bar z_{t_i}-x_{t_i}|\leq \big(Cn^{-1}t_i+\varepsilon \sigma_1\sqrt{t_i}\big)\exp\Big(\frac{M\sigma_1^2}{\sigma_0^2}t_i\Big).$$
\end{itemize}
Putting all the pieces together we obtain $\E_{P_f^{\xi}}| A_T(x)-\bar  A_T^n(x)|=O\Big(\frac{1}{n}+\varepsilon\Big)$.
\end{proof}

\begin{prop}\label{prop3}
 Under the same hypotheses of Proposition \ref{prop1}, $\Delta(\mo_{\xi}^{A_T(x)},\mo_{\xi}^{S_T^n(x)})\to 0$ and $\Delta(\mo_{\xi}^{S_T^n(x)},\mo_{\xi}^{\bar A_T^n(x)})\to 0$ as $n\to\infty$.
\end{prop}
\begin{proof}
 We shall prove only the first statement, the proof of the second one being identical. Since $\Ci_{S_T^n(x)}\subset\Ci_{ A_T(x)}$, it is clear that $\delta(\mo_{\xi}^{A_T(x)},\mo_{\xi}^{S_T^n(x)})=0$. To control $\delta(\mo_{\xi}^{S_T^n(x)},\mo_{\xi}^{A_T(x)})$ we will introduce the following Markov kernel $K^n$:
$$K^n(\omega,A):=\E_{P_0^{\xi}}\big(\I_A|_{\Ci_{S_T^n(x)}}\big)(\omega),\quad \forall A\in \Ci_{ A_T(x)}, \omega\in C,$$
where $P_0^{\xi}$ is defined as $P_f^{\xi}$ with $f \equiv 0$. Remark that the Markov kernel $K^n$ thus constructed coincides with the Markov kernel $N$ defined in \cite{C14}, Proposition 6.2, when $\varepsilon\equiv 1$. Making the same computations as in the cited proposition, we obtain that
\begin{align*}
\big\|K^nP_f^{\xi}|_{\Ci_{S_T^n(x)}}-P_f^{\xi}|_{\Ci_{ A_T(x)}}\big\|_{TV}
&\leq \frac{1}{2}\sqrt{\E_{P_f^{\xi}|\Ci_{ A_T(x)}}\int_{S_T^n(x)}^{ A_T(x)}\frac{f^2(x_r)}{\sigma^4(x_r)}dr}\\
&\leq\frac{M}{\sqrt 2\sigma_0^2}\sqrt{\E_{P_f^{\xi}|\Ci_{ A_T(x)}}\bigg(| A_T(x)-\bar A_T^n(x)|+\int_0^{| A_T(x)-\bar A_T^n(x)|}x_r^2dr\bigg)}\\
&=O\Big(\big(\E_{P_f^{\xi}|\Ci_{ A_T(x)}}| A_T(x)-\bar A_T^n(x)|\big)^{1/4}\Big).
\end{align*}
We then conclude that $\Delta(\mo_{\xi}^{A_T(x)},\mo_{\xi}^{S_T^n(x)})\to 0$ by means of Lemma \ref{lemma2}.
\end{proof}

\textbf{Step 4.} Using Steps 1--3 and the triangular inequality, one can find that $\Delta(\mo_{y}^{T},\mo_{\bar y}^{n,T})=O\Big(\frac{1}{\varepsilon n}+(n^{-1}+\varepsilon)^{1/4}\Big)$. Hence, to conclude the proof of Theorem \ref{teo1}, we only need to show the following proposition:

\begin{prop}\label{fisher}
 Under the same hypotheses of Proposition \ref{prop1}, $\Delta(\mo_{\bar y}^{n,T},\Qi_Z^n)=0$, for all $n$. 
\end{prop}
\begin{proof}
Note that, by using the Girsanov theorem, we can show that the measure $P_f^{n,\bar y}|\Ci_T$ is absolutely continuous with 
respect to $P_0^{n,\bar y}$ and the density is given by
$$\frac{dP_f^{n,\bar y}}{d P_0^{n,\bar y}}|_{\Ci_T}(\omega)=\exp\bigg(\sum_{i=0}^{n-1}\Big(\frac{f(\omega_{t_i})}{\varepsilon^2\sigma^2(\omega_{t_i})}(\omega_{t_{i+1}}-\omega_{t_i})-
 \frac{f^2(\omega_{t_i})}{2n\varepsilon^2\sigma^2(\omega_{t_i})}\Big)\bigg).$$

Hence, by means of the Fisher's factorization theorem, we can deduce that the application $S:\omega\to (\omega_{t_1},\dots,\omega_{t_n})$ is a sufficient statistic for the family of probability
measures $\{ P_f^{n,\bar y}|_{\Ci_T}; f\in \F\}$. We complete the proof remarking that the distribution of $(x_{t_1},\dots,x_{t_n})$ under $P_f^{n,\bar y}$ is the same as the one of $(Z_1,\dots,Z_n)$ under $\p$ and finally invoking the following property of the Le Cam distance (see Le Cam \cite{lecam}):

\emph{ Let $\mo_i=(\X_i,\A_i,\{P_{i,\theta}, \theta\in\Theta\})$, $i=1,2$, be two statistical models and let $(\X_1,\A_1)$ be a Polish space. 
Let $S:\X_1\to\X_2$ be a sufficient statistics
such that the distribution of $S$ under $P_{1,\theta}$ is equal to $P_{2,\theta}$. Then $\Delta(\mo_1,\mo_2)=0$. 
}
\end{proof}
\subsection{Proof of Theorem \ref{teo2}}
We will proceed in three Steps.

\textbf{Step 1.} Let us consider the application $F:\R\to \R$ defined as $F(x)=\int_0^{x}\frac{1}{\varepsilon\sigma(u)}du$. Remark that $F$ is well defined and one to one. Using the Itô formula, we have that
$$F(y_t)=F(w)+\int_0^t\Big(\frac{f(y_s)}{\varepsilon\sigma(y_s)}-\frac{\varepsilon\sigma'(y_s)}{2}\Big)ds+W_t.$$
Thus, if we set $\mu_t:=F(y_t)$, the new process $(\mu_t)$ satisfies the following SDE:
\begin{equation}\label{mu}
 \mu_0=F(w);\quad d\mu_t=\Big(\frac{f(F^{-1}(\mu_t))}{\varepsilon\sigma(F^{-1}(\mu_t))}-\frac{\varepsilon\sigma'(F^{-1}(\mu_t))}{2}\Big)dt+dW_t,\quad t\in[0,T].
\end{equation}
Observe that, thanks to hypotheses (H1), (H2) and the Lipschitz condition on $\frac{f}{\sigma}(\cdot)$, the drift function $b(x):=\frac{f(F^{-1}(x))}{\varepsilon\sigma(F^{-1}(x))}-\frac{\varepsilon\sigma'(F^{-1}(x))}{2}$ is such that $|b(0)|\leq \frac{M}{\varepsilon\sigma_0}+\frac{\varepsilon M}{2}$ and it is also Lipschitz:
\begin{align*}
 |b(x)-b(y)|&=\bigg|\Big(\frac{f(F^{-1}(x))}{\varepsilon\sigma(F^{-1}(x))}-\frac{\varepsilon\sigma'(F^{-1}(x))}{2}\Big)-\Big(\frac{f(F^{-1}(y))}{\varepsilon\sigma(F^{-1}(y))}-\frac{\varepsilon\sigma'(F^{-1}(y))}{2}\Big)\bigg|\\
    &\leq\frac{1}{\varepsilon}\bigg|\frac{f(F^{-1}(x))}{\sigma(F^{-1}(x))}-\frac{f(F^{-1}(y))}{\sigma(F^{-1}(y))}\bigg|+\varepsilon\bigg|\frac{\sigma'(F^{-1}(x))}{2}-\frac{\sigma'(F^{-1}(y))}{2}\bigg|\\
    &\leq \frac{L}{\varepsilon}|F^{-1}(x)-F^{-1}(y)|+\frac{M\varepsilon}{2}|F^{-1}(x)-F^{-1}(y)|\\
    &=\bigg|\int_x^y \varepsilon\sigma(F(u))du\bigg|\Big(\frac{L}{\varepsilon}+\frac{M\varepsilon}{2}\Big)\leq\Big(\frac{L}{\varepsilon}+\frac{M\varepsilon}{2}\Big)\sigma_1\varepsilon|x-y|;
    \end{align*}
In particular the existence and the uniqueness of a strong solution $\mu$ for the SDE \eqref{mu} are guaranteed. Let us denote by $P_f^{\mu}$ (resp. $Q_f^{n,\mu}$) the law of $\mu$ (resp. $(\mu_{t_1},\dots,\mu_{t_n})$) and introduce the statistical models
\begin{equation*}
 \mo_{\mu}^{T}=\big(C,\Ci_T,(P_f^{\mu}, f\in \F)\big),\quad \mathscr{Q}_{\mu}^{n}=\big(\R^n,\B(\R^n),(Q_f^{n,\mu}, f\in \F)\big).
\end{equation*}
By construction, $\mo_{\mu}^{T}$ (resp. $\mathscr{Q}_{\mu}^{n}$) is the image experiment of $\mo_y^{T}$ (resp. $\mathscr{Q}_{y}^{n}$) by $F$. Thus we have:
\begin{prop}\label{prop:mu}
 Under the hypotheses of Theorem \ref{teo2}, the statistical models $\mo_y^{T}$ (resp. $\mathscr{Q}_{y}^{n}$) and $\mo_{\mu}^{T}$ (resp. $\mathscr{Q}_{\mu}^{n}$) are equivalent.
\end{prop}
\textbf{Step 2.} Using the same notations as above, define a new drift function $\bar b_n:$ 
$$\bar b_n(t,\omega)=\sum_{i=0}^{n-1}b(\omega_{t_i})\I_{(t_i,t_{i+1}]}(t),\quad \forall \omega \in C,\ t\in[0,T]$$
and consider the diffusion process $(\bar\mu_t)$ on $(C,\Ci_T)$ having drift function given by $\bar b_n$ and diffusion coefficient equal to $1$, i.e.
\begin{equation}\label{barmu}
 \bar\mu_0=F(w);\quad d\bar\mu_t=\bar b_n(t,\bar\mu)dt+dW_t,\quad t\in[0,T].
\end{equation}
Denote by $P_f^{n,\bar\mu}$ the law of the solution of \eqref{barmu} and introduce the corresponding statistical model:
\begin{equation*}
 \mo_{\bar\mu}^{n,T}=\big(C,\Ci_T,(P_f^{n,\bar\mu}, f\in \F)\big).
\end{equation*}
\begin{prop}\label{prop:barmu}
 Under the hypotheses of Theorem \ref{teo2}, the statistical models $\mo_\mu^{n,T}$ and $\mo_{\bar\mu}^{n,T}$ are asymptotically equivalent as $n$ goes to infinity and $T$ is fixed.
\end{prop}
\begin{proof}
 One can use the same arguments as in the proof of Proposition \ref{prop2} obtaining a first bound given by
 $$\Delta(\mo_\mu^{T},\mo_{\bar\mu}^{n,T})\leq \sup_{f\in\F} \|P_f^{\mu}-P_f^{n,\bar\mu}\|_{TV}\leq 4 \sup_{f\in\F}\sqrt{\E_{P_f^{\mu}}\frac{1}{8}\int_0^T(b(x_s)-\bar b_n(s,x))^2ds}.$$
 Now, thanks to the $\tilde L$-Lipschitz character of $b$, one can write $\int_0^T(b(x_s)-\bar b_n(s,x))^2ds\leq \tilde L\sum_{i=0}^{n-1}\int_{t_i}^{t_{i+1}}(x_s-x_{t_i})^2ds$ so that the usual computations yield $\Delta(\mo_\mu^{T},\mo_{\bar\mu}^{n,T})=O(n^{-1})$.
\end{proof}
\textbf{Step 3.} Consider now the statistical model associated with the discrete observations $(\bar\mu_{t_1},\dots,\bar\mu_{t_n})$:
$$\mathscr{Q}_{\bar\mu}^{n}=\big(\R^n,\B(\R^n),(Q_f^{n,\bar\mu}, f\in \F)\big),$$
where $Q_f^{n,\bar\mu}$ denotes the law of the vector $(\bar\mu_{t_1},\dots,\bar\mu_{t_n})$.
\begin{prop}\label{prop:mud}
  Under the hypotheses of Theorem \ref{teo2}, we have
  $$\Delta(\mo_{\bar\mu}^{n,T},\mathscr{Q}_{\bar\mu}^{n})=0,\quad \Delta(\mathscr{Q}_{\bar\mu}^{n},\mathscr{Q}_{\mu}^{n})=O\big(n^{-1}\big),\ \forall n.$$
\end{prop}
\begin{proof}
 The first equivalence can be proved by means of a sufficient statistic as in the proof of Proposition \ref{fisher}; the second one follows directly from Step 2 since $\|Q_f^{n,\mu}-Q_f^{n,\bar\mu}\|_{TV}\leq \|P_f^{\mu}-P_f^{n,\bar\mu}\|_{TV}$ as we are only restricting to a smaller $\sigma$-algebra.
\end{proof}
 
 \appendix
\section{Background on Le Cam's theory}\label{LC}
\subsection{Asymptotic equivalence in the sense of Le Cam}
A \emph{statistical model} is a triplet $\mo_j=(\X_j,\A_j,\{P_{j,\theta}; \theta\in\Theta\})$ where $\{P_{j,\theta}; \theta\in\Theta\}$ 
is a family of probability distributions all defined on the same $\sigma$-field $\A_j$ over the \emph{sample space} $\X_j$ and $\Theta$ is the \emph{parameter space}.
The \emph{deficiency} $\delta(\mo_1,\mo_2)$ of $\mo_1$
with respect to $\mo_2$ quantifies ``how much information we lose'' by using $\mo_1$ instead of $\mo_2$ and is defined as
$\delta(\mo_1,\mo_2)=\inf_K\sup_{\theta\in \Theta}||KP_{1,\theta}-P_{2,\theta}||_{TV},$
 where TV stands for ``total variation'' and the infimum is taken over all ``transitions'' $K$ (see \cite{lecam}, page 18).
 In our setting, however, the general notion of ``transitions'' can be replaced with the notion of Markov kernels. Indeed, when the model $\mo_1$ is dominated 
and the sample space $(\X_2,\A_2)$ of the experiment $\mo_2$
is a Polish space, the infimum appearing on the definition of the deficiency $\delta$ can be taken over all Markov kernels $K$ on $\X_1\times \A_2$ 
(see \cite{N96}, Proposition 10.2), i.e.
\begin{equation}\label{eq:delta}
 \delta(\mo_1,\mo_2)=\inf_K\sup_{\theta\in \Theta}\sup_{A\in \A_2}\bigg|\int_{\X_1}K(x,A)P_{1,\theta}(dx)-P_{2,\theta}(A)\bigg|.
\end{equation}
The experiment $KP_{1,\theta}=(\X_1,\A_1,\{KP_{1,\theta}\}_{\theta\in \Theta})$ is called a \emph{randomization} of $\mo_1$ by the kernel $K$. If the kernel is deterministic, i.e. for $T:(\X_1,\A_1)\to (\X_2,\A_2)$ a random variable, $T(x,A):=\I_{A}(T(x))$, the experiment $T\mo_1$ is called the \emph{image experiment by the random variable $T$}.
Closely associated with the notion of deficiency is the so called $\Delta$-distance, i.e. the pseudo metric defined by:
$$\Delta(\mo_1,\mo_2):=\max(\delta(\mo_1,\mo_2),\delta(\mo_2,\mo_1)).$$
The sufficiency of a statistic can be expressed in terms of the $\Delta$-distance. More precisely, the following holds (see \cite{C14}, Proposition 8.1, page 23).
\emph{Let $T:(\X_1,\A_1)\to (\X_2,\A_2)$ be a random variable. The statistic $T$ is sufficient for $\mo_1$ if and only if $\Delta(\mo_1,T\mo_1)=0.$}

Also, remark that thanks to \eqref{eq:delta}, if $\mo_1=(\X,\A_1,\{P_\theta;\theta\in\Theta\})$ and $\mo_2=(\X,\A_2,\{P_\theta;\theta\in\Theta\})$ with $\A_2\subset\A_1$, then $\delta(\mo_1,\mo_2)=0$.

Two sequences of statistical models $(\mo_{1}^n)_{n\in\N}$ and $(\mo_{2}^n)_{n\in\N}$ are called \emph{asymptotically equivalent}
if $\Delta(\mo_{1}^n,\mo_{2}^n)$ tends to zero as $n$ goes to infinity. Similarly, the statistic $T^n$ is \emph{asymptotically sufficient} for $\mo_1^n$ if  $\Delta(\mo_{1}^n,T^n\mo_{1}^n)$ tends to zero as $n$ goes to infinity.

\subsection*{Acknowledgements}

I would like to thank Valentine Genon-Catalot for several interesting discussions, especially in suggesting to taking into account the relation between diffusion processes with small variance and deterministic limits. Also, I would like to give a special thank to Pierre Étoré, with whom a lot of hours were spent discussing different approaches to the proof of Lemma \ref{lemma2}. More generally, I am very grateful for all the time he has invested in supervising this project.
\printbibliography
\end{document}